\documentclass[letterpaper,10pt]{amsart}

\usepackage{amsmath}
\usepackage{amssymb}
\usepackage[T1]{fontenc}
\usepackage[utf8]{inputenc}
\usepackage[english]{babel}
\usepackage{hyperref}
\usepackage{tikz}
\usepackage{xspace}
\usepackage{xcolor}

\theoremstyle{definition}
\newtheorem{defn}{Definition}[section]

\newtheorem{question}[defn]{Question}
\newtheorem{remark}[defn]{Remark}
\theoremstyle{plain}
\newtheorem{lemma}[defn]{Lemma}
\newtheorem{prop}[defn]{Proposition}
\newtheorem{thm}[defn]{Theorem}

\newcommand{\diam}{\operatorname{diam}}
\newcommand{\U}{\mathcal{U}}
\newcommand{\tame}{\mathrm{tame}}

\newcommand{\Homeo}{\mathrm{Homeo}}
\newcommand{\Wazewski}{Wa\.{z}ewski\xspace}
\newcommand{\End}{\mathrm{End}}
\newcommand{\Id}{\mathrm{Id}}
\newcommand{\osc}{\mathrm{osc}}
\DeclareMathOperator{\rk}{\mathrm{rk}}

\begin{document}
	
\title[Some examples of tame dynamical systems]{Some examples of tame dynamical systems answering questions of Glasner and Megrelishvili}
\author{Alessandro Codenotti}
\address{A. Codenotti, Institut f\"{u}r Mathematische Logik und Grundlagenforschung, Universit\"{a}t M\"{u}nster, Einsteinstrasse 62, 48149 M\"{u}nster, Germany}
\email{acodenot@uni-muenster.de}

\thanks{Funded by the Deutsche Forschungsgemeinschaft (DFG, German Research Foundation) under Germany’s Excellence Strategy EXC 2044–390685587, Mathematics Münster: Dynamics– Geometry–Structure and by CRC 1442 Geometry: Deformations and Rigidity.}


\begin{abstract}
	Glasner and Megrelishvili proved that every continuous action of a topological group $G$ on a dendrite $X$ is tame. We produce two examples of an action on a dendrite which is not $\tame_1$, answering a question they raised. We then show that actions on dendrites have $\beta$-rank at most $2$ and produce examples of tame metric dynamical systems of $\beta$-rank $\alpha$ for any $\alpha<\omega_1$, answering another question of Glasner and Megrelishvili. 
\end{abstract}

\maketitle

\section{Introduction}

A \emph{dynamical system} is a continuous action $\alpha\colon G\times X\to X$ of a topological group $G$ on a compact space $X$. We will usually suppress the action from the notation and simply refer to $(X,G)$ or $G\curvearrowright X$ as a dynamical system. Given a dynamical system $(X,G)$ with a faithful action $G\curvearrowright X$, as will always be the case in this paper, we can identify $G$ with a subset of $X^X$, through the map $g\mapsto(x\mapsto gx)$. The \emph{Ellis semigroup} of $(X,G)$, denoted by $E(X,G)$, is the closure of $G$ in $X^X$, where the latter is equipped with the pointwise convergence topology. It is a compact right topological semigroup whose algebraic properties encode dynamical properties of the system $(X,G)$. See \cite{auslander} for more information about this object.

Tame dynamical systems were introduced by K\"{o}hler \cite{Kohler} under the name of \emph{regular} systems and have been studied by various authors \cite{KerrLi}, \cite{Glasner1}, \cite{Glasner2},\cite{Glasner3}, \cite{Huang}, \cite{GMtrees}, see \cite{Glasner4} for an overview and more references. There are various characterizations of tame dynamical systems in the aforementioned papers, but the one we will use is that a metric dynamical system $(X,G)$ is \emph{tame} if and only if the Ellis semigroup $E(X,G)$ is a Rosenthal compactum.
Recall that a function $f\colon X\to Y$ between Polish spaces is a \emph{Baire class 1 function} if $f^{-1}(U)$ is $F_\sigma$ in $X$ for every open $U\subseteq Y$. A topological space $K$ is called a \emph{Rosenthal compactum} if it is homeomorphic to a compact subspace of $B_1(X,\mathbb R)$, the space of Baire class 1 functions $X\to\mathbb R$, for some Polish space $X$.
We will also make use of the equivalent characterization proved in \cite[Theorem 6.3]{GlasnerMegrelishviliUspenskij}, namely that a metric dynamical system $(X,G)$ is tame if and only if every function in its Ellis semigroup is a Baire class 1 function $X\to X$.

Recently Glasner and Megrelishvili introduced in \cite{GM} two measures of complexity of a tame dynamical system. The first, based on a Trichotomy of Todor\v{c}evi\'{c} \cite{TodorcevicTrichotomy} for Rosenthal compacta, is a division of tame systems into three subclasses as follows.

\begin{defn} Let $G$ be a topological group and $G\curvearrowright X$ a $G$-flow. We say that this dynamical system is:
	\begin{itemize}
        \item $\tame$ if the Ellis semigroup $E(X,G)$ is a Rosenthal compactum.
		\item $\tame_1$ if the Ellis semigroup $E(X,G)$ is a first countable Rosenthal compactum.
		\item $\tame_2$ if the Ellis semigroup $E(X,G)$ is an hereditarily separable Rosenthal compactum.
	\end{itemize}
\end{defn}

Every $\tame_2$ system is $\tame_1$ and every $\tame_1$ system is $\tame$, but it was proved in \cite{GM} that those inclusions are strict.

A large class of examples of tame dynamical systems is provided by \cite{GMtrees}, in which Glasner and Megrelishvili showed that every dynamical system $(X,G)$ where $X$ is a dendrite is tame. It is natural to ask whether such a dynamical system must always be $\tame_1$ and indeed this was raised as Question 12.2 in \cite{GM}. In the same paper the authors showed that the answer is positive for the simplest dendrite $X=[0,1]$ and more generally for all linearly ordered metric spaces (on a dendrite $X$, modulo fixing an arbitrary root $r\in X$, there is a natural partial order which is closed in $X^2$, namely the order $x\leq y$ if and only if $x$ belongs to the arc from $r$ to $y$). In Section \ref{section: Wazewski dendrites have actions that are not tame_1} we give a negative answer to this question, in particular we show in Theorem \ref{thm: Wazewski is not tame1} that the action $\Homeo(W_P)\curvearrowright W_P$, where $W_P$ is a \Wazewski dendrite, is not $\tame_1$.
A dynamical system $(X,G)$ is called \emph{rigid} if there is a net $(p_\lambda)_{\lambda\in\Lambda}$ in $E(X,G)\setminus G$ such that $p_\lambda\to\mathrm{Id}_X$. We verify that the action $\Homeo(W_P)\curvearrowright W_P$ is rigid in Lemma \ref{lemma: Wazewski is rigid}, which gives an example of a tame, rigid, metric system which is not hereditarily nonsensitive as well as an example of a tame, rigid, metric, minimal system which is not equicontinuous. This answers positively question 5.6 of \cite{GM}, which asks whether such examples exist. In Section \ref{section: the second example} we construct another example of a dendrite $X$ such that the action $\Homeo(X)\curvearrowright X$ is not $\tame_1$, as far as we know this space has not appeared in the literature before.

The second measure of complexity for metric tame dynamical systems is based on the theory of ordinal ranks for Baire class 1 functions developed by Kechris and Louveau \cite{KechrisLouveau}. It is a countable ordinal associated to the dynamical system $(X,G)$, called its $\beta$-rank, denoted $\beta(X,G)$. Glasner and Megrelishvili pointed out that all the known examples of metric tame dynamical systems $(X,G)$ satisfy either $\beta(X,G)=1$ or $\beta(X,G)=2$ and asked (Question 11.8 of \cite{GM}) whether, for every $\alpha<\omega_1$, there is a metric tame dynamical system $(X_\alpha,G_\alpha)$ with $\beta(X_\alpha,G_\alpha)=\alpha$. In Section \ref{section: actions on dendrites have low beta-rank} we prove that there are no such examples among actions on dendrites, in particular we prove in Theorem \ref{thm: actions on dendrites have low beta-rank} that if $(X,G)$ is a dynamical system with $X$ a dendrite, then $\beta(X,G)\leq 2$. In Section \ref{section: systems with arbitrary beta-rank} we construct, for every $\alpha<\omega_1$, a metric tame dynamical system with $\beta$-rank exactly $\alpha$, thus giving a positive answer to the aforementioned question. 

We conclude with a few open questions and comments in Section \ref{section: open questions and conclusion}.

\section{\Wazewski dendrites have actions that are not \texorpdfstring{$\tame_1$}{tame 1}}
\label{section: Wazewski dendrites have actions that are not tame_1}

Glasner and Megrelishvili proved in \cite{GMtrees} that whenever $X$ is a dendrite and $G$ is a topological group acting continuously on $X$, the action is tame. The question of whether such an action is always $\tame_1$ was raised as Question 12.2 in \cite{GM}. We answer it negatively with an argument based on the following result of Ellis from \cite{Ellis}. The same approach has been already successfully applied in \cite{GM} to prove that the dynamical system $\Homeo_+(S^1)\curvearrowright
 S^1$ is tame but not $\tame_1$. 

\begin{prop}\label{prop: Ellis}
	Let $(X,G)$ be a proximal minimal system. Suppose that there is $a\in X$ and an uncountable set $\{b_i\}_{i\in I}\subseteq X\setminus\{a\}$ such that for every $i\in I$ the function $p_{a,b_i}$ defined by $$p_{a,b_i}(x)=\begin{cases}
	b_i & x=b_i \\ a & \text{otherwise}\end{cases}$$
	is an idempotent in $E(X,G)$. Then the idempotent $p_a$ with constant value $a$ is in $E(X,G)$ and it doesn't admit a countable neighbourhood basis. In particular $E(X,G)$ is not first countable.
\end{prop} 

\begin{defn}
     A \emph{dendrite} $X$ is a compact, connected, locally connected, metrizable topological space, such that for every $x,y\in X$, there is a unique arc in $X$ with endpoints $x$ and $y$.
\end{defn}

For more equivalent characterizations of dendrites and results about them we refer to \cite[Chapter 10]{Nadler}.
 
We will show that for any $P\subseteq\{3,4,\ldots,\omega\}$ the natural action $\Homeo(W_P)\curvearrowright W_P$ on the \Wazewski dendrite $W_P$ (defined below) is not $\tame_1$ by showing that the hypothesis of Proposition \ref{prop: Ellis} are satisfied. We will make repeated use of the following result of Charatonik and Dilks \cite{CD}.

\begin{prop}\label{prop: characterization of W_P}
	Fix $P\subseteq\{3,4,\ldots,\omega\}$. If $X$ and $Y$ are two dendrites such that:
	\begin{itemize}
	    \item for every $p\in P$ the set of ramification points of order $p$ is arcwise dense in both $X$ and $Y$,
	    \item every ramification point of both $X$ and $Y$ has order in $P$,
	\end{itemize}
	then $X\cong Y$. In particular the \Wazewski dendrite $W_P$ is characterized by the two properties above. Suppose moreover that we have two fixed countable dense sets $E=\{e_1,e_2,\ldots\}\subseteq\End(X)$ and $F=\{f_1,f_2,\ldots\}\subseteq\End(Y)$. Then the homeomorphism $\varphi\colon X\to Y$ can be chosen to satisfy $\varphi(e_1)=f_1,\varphi(e_2)=f_2$ and $\varphi(E)=F$.
\end{prop} 

In order to apply Proposition \ref{prop: Ellis} we need to check that $\Homeo(W_P)\curvearrowright W_P$ is a proximal system. This follows at once from \cite[Theorem 10.1]{duchesne2018group}, but it can also be proved directly, which we do in the following lemma. First we introduce some notation. Given $x\in X$, where $X$ is a dendrite, we denote by $\widehat{x}$ the set of connected components of $X\setminus \{x\}$. We say that $x$ is an endpoint if $|\widehat{x}|=1$, a regular point if $|\hat{x}|=2$ and a ramification point if $|\widehat{x}|\geq 3$. Given $x\neq y\in X$ there is a unique $C\in\widehat{x}$ such that $y\in C$. We denote $C\cup\{x\}$ by $C_x(y)$ and we write $C_{x,y}$ for $C_x(y)\cap C_y(x)$. Note that, by Proposition \ref{prop: characterization of W_P}, if $X=W_P$, then $C_x(y)$ and $C_{x,y}$ are homeomorphic to $X$ for all $x\neq y\in X$. We denote by $[x,y]$ the unique arc between $x$ and $y$. We say that the arc $[x,y]$ in a dendrite $X$ is \emph{free} if $(x,y)=[x,y]\setminus\{x,y\}$ is open in $X$. A free arc $[x,y]$ is called \emph{maximal} if there is no free arc $[v,w]$ with $[x,y]\subsetneq[v,w]$.

\begin{lemma}
	Fix $P\subseteq\{3,4,\ldots,\omega\}$. The dynamical system $\Homeo(W_P)\curvearrowright W_P$ is proximal.
\end{lemma}
\begin{proof}
	Let $x\neq y\in W_P$ and $\varepsilon>0$. We want to find $h\in\Homeo(W_P)$ such that $d(h(x),h(y))<\varepsilon$. We can find a regular point $c\in W_P$ such that $C_c(x)=C_c(y)$ and the diameter of the only other (because $c$ is regular) component $C'$ of $W_P\setminus\{c\}$ is smaller than $\varepsilon$. Let $C=C'\cup\{c\}$. By the characterization in Proposition \ref{prop: characterization of W_P} we have $C\cong C_c(x)\cong W_P$ and since $c$ is an endpoint of both components we can find homeomorphisms $\varphi_1\colon C\to C_c(x)$ and $\varphi_2\colon C_c(x)\to C$ both of which map $c$ to itself. Now $h=\varphi_1\cup\varphi_2$, meaning that $h(z)=\varphi_1(z)$ if $z$ is in the domain of $\varphi_1$, and $h(z)=\varphi_2(z)$ if $z$ is in the domain of $\varphi_2$, is the desired homeomorphism. Indeed $h$ is well defined since the domains of $\varphi_1$ and $\varphi_2$ only intersect in $\{c\}$ and we have $\varphi_1(c)=c=\varphi_2(c)$, $h$ is defined everywhere since $W_P=C\cup C_c(x)$, it is continuous since it is continuous when restricted to the finite closed cover $\{C,C_c(x)\}$ (as it agrees with $\varphi_1$ on the former and with $\varphi_2$ on the latter), it is injective since $\varphi_1$ and $\varphi_2$ are, and it is surjective since the images of $\varphi_1$ and $\varphi_2$ cover $W_P$. 
\end{proof}

We can now prove that the dynamical system $\Homeo(W_P)\curvearrowright W_P$ satisfies the assumptions of Proposition \ref{prop: Ellis}.

\begin{lemma}\label{lemma: p_ab is in E(X,G)}
	Let $G=\Homeo(W_P)$ with the standard action $G\curvearrowright W_P$. For every pair $a\neq b\in\End(W_P)$ the idempotent $p_{a,b}$ is in $E(W_P,G)$.
\end{lemma}
\begin{proof}
	Fix $a\neq b\in\End(W_P)$. We want to show that every neighbourhood of $p_{a,b}$ in the product topology of $W_P^{W_P}$ meets $\Homeo(W_P)$. In other words given $\{x_1,\ldots,x_n\}\in W_P$ and $\varepsilon>0$ we want to find $h\in\Homeo(W_P)$ such that $$d(h(x_i),p_{a,b}(x_i))<\varepsilon\text{ for all }1\leq i\leq n.$$
	Let $z$ be a regular point such that $C_z(a)\neq C_z(b)$ and $\diam(C_z(a))<\varepsilon$. Let $y$ be a regular point such that $C_y(b)\cap\{x_1,\ldots,x_n\}\subseteq\{b\}$. By Proposition \ref{prop: characterization of W_P} we have $C_z(a)\cong C_y(a)$ and $C_y(b)\cong C_z(b)$ so that we can find homeomorphisms $\varphi_1\colon C_y(a)\to C_z(a)$ and $\varphi_2\colon C_y(b)\to C_z(b)$ such that \begin{align*}
	\varphi_1(a)=a\quad& \varphi_1(y)=z \\
	\varphi_2(b)=b\quad& \varphi_2(y)=z.
	\end{align*}
	Then $h=\varphi_1\cup\varphi_2$ is as desired. Indeed $h$ is well defined since $C_y(a)\cap C_y(b)=\{y\}$ and $\varphi_1(y)=z=\varphi_2(y)$, $h$ is defined everywhere since $W_P=C_y(a)\cup C_y(b)$, it is continuous since it is continuous on the finite closed cover $\{C_y(a),C_y(b)\}$, and it is both injective, since $\varphi_1$ and $\varphi_2$ are, and surjective, since the images of $\varphi_1$ and $\varphi_2$ cover $W_P$. Moreover if $x_i\neq b$, then $x_i\in C_y(a)$, so that $h(x_i)\in C_z(a)$, which implies that $d(p_{a,b}(x_i),h(x_i))=d(a,h(x_i))<\varepsilon$. If instead $x_i=b$, then $x_i\in C_y(b)$, so that $h(b)=\varphi_2(b)=b=p_{a,b}(b)$ and $d(p_{a,b}(x_i),h(x_i))=0$.
\end{proof}

Since $\End(W_P)$ is uncountable we see that the dynamical system $\Homeo(W_P)\curvearrowright W_P$ satisfies the assumptions of Proposition \ref{prop: Ellis} by fixing $a\in\End(W_P)$ and letting $b_i$ vary over $\End(W_P)\setminus\{a\}$. In particular its enveloping semigroup is not first countable, which immediately gives the following theorem, answering negatively Question 12.2 of \cite{GM}. 

\begin{thm}\label{thm: Wazewski is not tame1}
The dynamical system $\Homeo(W_P)\curvearrowright W_P$ is not $\tame_1$.
\end{thm}

Now that we know that the dynamical system $\Homeo(W_P)\curvearrowright W_P$ is not $\tame_1$, it provides an example answering another question from \cite{GM}, but first we need the following definition.

\begin{defn}
    A dynamical system $(X,G)$ is called \emph{weakly rigid} if there exists a net $g_\lambda$ in $E(X,G)\setminus\{\mathrm{Id_X}\}$ with $g_\lambda\to\Id_X$ pointwise. If instead there exists a sequence $g_i$ in $E(X,G)\setminus\{\mathrm{Id_X}\}$ with $g_i\to\Id_X$, the system is called \emph{rigid}. 
\end{defn}

Note that, by \cite[Theorem 3.2]{GMHNSsystems}, a tame dynamical system is weakly rigid if and only if it is rigid. 

In Question 5.6 of \cite{GM} Glasner and Megrelishvili asked whether there are examples of the following:
\begin{enumerate}
    \item A tame, rigid, metric system $(X,G)$ which is not hereditarily nonsensitive.
    \item A tame, rigid, minimal, metric system $(X,G)$ which is not equicontinuous. 
\end{enumerate}

We check that the dynamical system $(W_P,\Homeo(W_P))$ is rigid in the following lemma. Since it is minimal and tame but neither hereditarily nonsensitive (since its Ellis semigroup is not 1st countable by Theorem \ref{thm: Wazewski is not tame1} so in particular it is not metrizable), nor equicontinuous (since its Ellis semigroup contains discontinuous functions), it provides a positive answer to the question above.

\begin{lemma}\label{lemma: Wazewski is rigid}
    The dynamical system $(W_P,\Homeo(W_P))$ is rigid.
\end{lemma}
\begin{proof}
    Fix two sequences $(a_i)_{i<\omega},(b_i)_{i<\omega}$ in $W_P$ such that:
    \begin{itemize}
        \item For all $i<\omega$, $a_i$ and $b_i$ are regular points.
        \item There is a regular point $c$ such that $\lim_i a_i=\lim_i b_i=c$.
        \item For every $i<\omega$, $a_i$ and $b_i$ are in different components of $W_P\setminus\{c\}$.
    \end{itemize}
    Now let $g_i\colon W_P\to W_P$ be any homeomorphism which is not the identity on $C_{a_i,b_i}$, but which is the identity on the rest of $W_P$, which can be constructed by glueing together three maps as in the proof of Lemma \ref{lemma: p_ab is in E(X,G)}. Since $g_i\to\mathrm{Id}_{W_P}$ pointwise, this shows that $(W_P,\Homeo(W_P))$ is rigid.
\end{proof}

\section{Another example of a dendrite with an action that is not \texorpdfstring{$\tame_1$}{tame 1}}
\label{section: the second example}
In this section we construct another dendrite $X$ with the property that the action $\Homeo(X)\curvearrowright X$ is not $\tame_1$. We will describe the construction of a dendrite in which all ramification points have order either $3$ or $4$, but the same construction can be carried out for any $n\neq m$ to produce a dendrite with the same dynamical properties in which all ramification points have order either $n$ or $m$. It is also possible to do variations of this construction to produce dendrites with more than two orders for the ramification points, but we will not describe them there, since they don't seem to exhibit new dynamical behaviour. In order to describe the dendrite we will make use of the following two auxilliary spaces, each of which is homeomorphic to the unit interval $[0,1]$ but is augmented with some extra information. Let $T_0$ and $T_1$ be two copies of $[0,1]$ with fixed countable dense sets $D_0$ and $D_1$, such that, for all $n\in\omega$ and $i\in\{0,1\}$, $n/(n+1)\not\in D_i$. We now colour all points of $D_0\cap (n/(n+1),(n+1)/(n+2))$ red if $n$ is even and green if $n$ is odd. In $T_1$ we do the same, except that the colours are swapped with respect to the parity of $n$, see Figure \ref{fig: T_0 and T_1}. The intuitive idea is that red points of $D_i$ will become ramification points of order $4$ in $X$, while the green ones will become ramification points of order $3$.

\begin{figure}[h]
\begin{tikzpicture}[xscale=10,yscale=2]

    \definecolor{MyRed}{RGB}{222,6,26}
    \definecolor{MyGreen}{RGB}{79,180,67}
    
    \draw [line width=.1mm] (0,0) -- (1,0);
    \draw [line width=.5mm, color=MyRed] (0,0) -- (1/2,0);
    \draw [line width=.5mm, color=MyGreen] (1/2,0) -- (2/3,0);
    \draw [line width=.5mm, color=MyRed] (2/3,0) -- (3/4,0);
    \draw [line width=.5mm, color=MyGreen] (3/4,0) -- (4/5,0);
    \draw [line width=.5mm, color=MyRed] (4/5,0) -- (5/6,0);
    \draw [line width=.5mm, color=MyGreen] (5/6,0) -- (6/7,0);
    \draw [line width=.5mm, color=MyRed] (6/7,0) -- (7/8,0);
    \node [fill,circle,scale=0.4,label=above:{1}] at (1,0) {};
    \node [fill,circle,scale=0.4,label=above:{0}] at (0,0) {};
    \node [label=below:{$T_0$}] at (1/2,0) {};

    \draw [line width=.1mm] (0,1) -- (1,1);
    \draw [line width=.5mm, color=MyGreen] (0,1) -- (1/2,1);
    \draw [line width=.5mm, color=MyRed] (1/2,1) -- (2/3,1);
    \draw [line width=.5mm, color=MyGreen] (2/3,1) -- (3/4,1);
    \draw [line width=.5mm, color=MyRed] (3/4,1) -- (4/5,1);
    \draw [line width=.5mm, color=MyGreen] (4/5,1) -- (5/6,1);
    \draw [line width=.5mm, color=MyRed] (5/6,1) -- (6/7,1);
    \draw [line width=.5mm, color=MyGreen] (6/7,1) -- (7/8,1);
    \node [fill,circle,scale=0.4,label=above:{1}] at (1,1) {};
    \node [fill,circle,scale=0.4,label=above:{0}] at (0,1) {};
    \node [label=below:{$T_1$}] at (1/2,1) {};
\end{tikzpicture}
\caption{$T_0$ and $T_1$, with the first few intervals of the form $(n/(n+1),(n+1)/(n+2))$. Points that will become of order $4$ during the construction are coloured in red, while points that will become of order $3$ are coloured in green.}
\label{fig: T_0 and T_1}
\end{figure}

To construct $X$ let $X_0$ be four copies of $T_0$ identified together in $0$. Now given $X_i$ we construct $X_{i+1}$ by attaching more copies of $T_0$ or $T_1$ until there are no free arcs left in $X_i$. For each maximal free arc $A$ in $X_i$, which will necessarily be a copy of either $T_0$ or $T_1$ by construction and so has a fixed countable dense set $D(A)$ we do the following. For each $a\in D(A)$ let $T(a)=T_0\vee T_0$ is $a$ is red, where the two copies of $T_0$ are joined in $0$, and $T(a)=T_1$ if $a$ is green. To each $a\in D(A)$ we attach a copy of $T(a)$ such that:
\begin{itemize}
    \item The point $a$ is the point labelled $0$ of $T(a)$.
    \item $X_i\cap T(a)=\{a\}$.
    \item $T(a_1)\cap T(a_2)=\varnothing$ for all pairs $a_1\neq a_2\in D(a)$.
    \item For every $\varepsilon>0$, there are only finitely many $a\in D(a)$ such that $\diam(T(a))>\varepsilon$.
    \item For every $a\in D(A)$, $\diam(T(a))<2^{-i}$.
\end{itemize}

See Figure \ref{fig: X_1} for a drawing of $X_1$. Note that every $X_i$ is a dendrite. Moreover we define a monotone (meaning that the preimage of a connected set is connected) surjection $f^{i+1}_i\colon X_{i+1}\to X_i$ so that for every maximal free arc $A$ of $X_i$ and every $a\in D(A)$, $f^{i+1}_i(T(a))=a$, while $f^{i+1}_i$ is the identity when restricted to $X_i$.

We define $X=\varprojlim X_i$ which is a dendrite by Theorem 10.36 of \cite{Nadler} and, thanks to the last condition in the construction above, the hypothesis of the Anderson-Choquet embedding theorem \cite[Theorem 2.10]{Nadler} are satisfied, so that $$X\cong \overline{\bigcup_{i=1} X_i},$$ when the spaces $X_i$ are realized as an increasing chain of subspaces of the plane. Note that, since neither the ramification point of order $3$ nor those of order $4$ are arcwise dense in $X$, the dendrite $X$ is not homeomorphic to $W_{\{3,4\}}$.

\newcommand{\drawlineh}[4]{%
\draw[line width=0.35mm, black] (#1,#2) -- (#3,#4);
\draw[line width=0.5mm] (#1,#2) edge (1/2*#3,#4) edge (2/3*#3,#4) edge (3/4*#3,#4) edge (4/5*#3,#4) edge (5/6*#3,#4) edge (6/7*#3,#4);
}

\definecolor{MyRed}{RGB}{222,6,26}
\definecolor{MyGreen}{RGB}{79,180,67}

\newcommand{\drawlinev}[4]{%
\draw[line width=0.35mm, black] (#1,#2) -- (#3,#4);
\draw[line width=0.5mm] (#1,#2) edge (#3,1/2*#4) edge (#3,2/3*#4) edge (#3,3/4*#4) edge (#3,4/5*#4) edge (#3,5/6*#4) edge (#3,6/7*#4);
}

\def\alternatecolorred{%
    \pgfkeysalso{MyRed}%
    \global\let\alternatecolor\alternatecolorblue 
}
\def\alternatecolorblue{%
    \pgfkeysalso{MyGreen}%
    \global\let\alternatecolor\alternatecolorred 
}
\let\alternatecolor\alternatecolorred 

\begin{figure}[h]
\begin{tikzpicture}[xscale=6, yscale=6, every edge/.append code = {%
    \global\let\currenttarget\tikztotarget 
    \pgfkeysalso{append after command={(\currenttarget)}}
    \alternatecolor
}]

\drawlineh{0}{0}{1}{0}
\drawlineh{0}{0}{-1}{0}
\drawlinev{0}{0}{0}{1}
\drawlinev{0}{0}{0}{-1}

\drawlineh{0}{1/4}{1/5}{1/4}
\drawlineh{0}{1/4}{-1/5}{1/4}

\drawlinev{1/4}{0}{1/4}{1/5}
\drawlinev{1/4}{0}{1/4}{-1/5}

\drawlineh{0}{-1/4}{1/5}{-1/4}
\drawlineh{0}{-1/4}{-1/5}{-1/4}

\drawlinev{-1/4}{0}{-1/4}{1/5}
\drawlinev{-1/4}{0}{-1/4}{-1/5}

\drawlineh{0}{-1/8}{1/10}{-1/8}
\drawlineh{0}{-1/8}{-1/10}{-1/8}

\drawlineh{0}{-3/8}{1/10}{-3/8}
\drawlineh{0}{-3/8}{-1/10}{-3/8}

\drawlinev{-1/8}{0}{-1/8}{1/10}
\drawlinev{-1/8}{0}{-1/8}{-1/10}

\drawlinev{-3/8}{0}{-3/8}{1/10}
\drawlinev{-3/8}{0}{-3/8}{-1/10}

\drawlinev{3/8}{0}{3/8}{1/10}
\drawlinev{3/8}{0}{3/8}{-1/10}

\drawlinev{1/8}{0}{1/8}{1/10}
\drawlinev{1/8}{0}{1/8}{-1/10}

\drawlineh{0}{1/8}{1/10}{1/8}
\drawlineh{0}{1/8}{-1/10}{1/8}

\drawlineh{0}{3/8}{1/10}{3/8}
\drawlineh{0}{3/8}{-1/10}{3/8}

\let\alternatecolor\alternatecolorblue
\drawlineh{0}{13/24}{1/18}{13/24}
\drawlineh{0}{7/12}{1/12}{7/12}
\drawlineh{0}{15/24}{1/18}{15/24}

\drawlineh{0}{-13/24}{-1/18}{-13/24}
\drawlineh{0}{-7/12}{-1/12}{-7/12}
\drawlineh{0}{-15/24}{-1/18}{-15/24}

\drawlinev{13/24}{0}{13/24}{-1/18}
\drawlinev{7/12}{0}{7/12}{-1/12}
\drawlinev{15/24}{0}{15/24}{-1/18}

\drawlinev{-13/24}{0}{-13/24}{1/18}
\drawlinev{-7/12}{0}{-7/12}{1/12}
\drawlinev{-15/24}{0}{-15/24}{1/18}

\end{tikzpicture}
\caption{(An approximation of) the stage $X_1$ of the construction. Points that have order $4$ or will become points of order $4$ in $X_2$ are coloured in red, while points that have order $3$ in $X_1$ or will become points of order $3$ in $X_2$ are coloured in green.}
\label{fig: X_1}
\end{figure}

\begin{defn}\label{def: three kind of endpoints}
    We distinguish three types of endpoints in $X$ as follows. Given $x\in\End(X)$ we say that $x$ is 
    \begin{itemize}
        \item \emph{Green}, if for every arc $\gamma\colon[0,1]\to X$ with $\gamma(1)=x$, there is $r\in[0,1]$, such that all branching points of $X$ contained in $\gamma([r,1])$ have order $3$.
        \item \emph{Red}, if for every arc $\gamma\colon[0,1]\to X$ with $\gamma(1)=x$, there is $r\in[0,1]$, such that all branching points of $X$ contained in $\gamma([r,1])$ have order $4$.
        \item \emph{Alternating} otherwise.
    \end{itemize}
    We will also call ramification points with the colour they've been assigned in the construction described above. In other words we will refer to ramification points of order $4$ as \emph{red} and to ramification points of order $3$ as \emph{green}.
\end{defn}

It is clear from the construction that alternating endpoints are dense in $X$, and it is not hard to check that the same holds for green and red endpoints. It is also easy to check that for each of the three types above, $X$ contains uncountably many endpoints of that type, which we do in the following lemma.

\begin{lemma}\label{lemma: uncountably many endpoints of each type}
    For each of the three types of endpoints introduced in Definition \ref{def: three kind of endpoints}, $X$ contains uncountably many endpoints of that type.
\end{lemma}
\begin{proof}
    We write the argument for alternating endpoints. We will construct an injection from $2^\omega$ to the set of alternating endpoints. In $X_0$ we have a fixed countable dense set of points that will become ramification points in $X_1$. Let $x_0$ and $x_1$ be red points from this set. In $X_1$ there is a unique ramification point $y_0$ of order $4$ with $f^1_0(y_0)=x_0$, let $C_0$ be a connected component of $X_1\setminus\{y_0\}$ such that $f^1_0(C_0)=x_0$. On $C_0$ we have a countable dense set of points that will become ramification points in $X_1$. Let $x_{00}$ and $x_{01}$ be two green points from this countable set, chosen so that both are at distance $\leq 1/2$ from an endpoint of $X_1$. Note that $f^1_0(x_{00})=f^1_0(x_{01})=x_0$. Analogously we can find two green points $x_{10}$ and $x_{11}$ that are close to an endpoint of $X_1$ and satisfy $f^1_0(x_{10})=f^1_0(x_{11})=x_1$. We can now iterate this construction, so suppose that for $n\in\omega$ we have constructed $2^n$ points in $X_{n-1}$, indexed over $s\in 2^n$ and such that $f^{n-1}_{n-2}(x_s)=x_{s'}$, where $s'$ is obtained from $s$ by dropping the last digit. For each $s\in 2^n$ let $y_s$ be the unique ramification point in $X_n$ with $f^n_{n-1}(y_s)=x_s$ and let $C_s$ be a connected component of $X_n\setminus\{y_s\}$ such that $f^n_{n-1}(C_s)=x_s$. There is a countable dense set of points in $C_s$ that will become ramification points in $X_{n+1}$. From this set we can choose $x_{s0}$ and $x_{s1}$ of colour opposite to that of $x_s$ and at distance $\leq 2^{-n+1}$ from an endpoint of $X_n$. It is now easy to check that the map  $2^\omega\to\varprojlim X_i$, $s\mapsto (x_n)_{n\in\omega}$ with $x_n=x_{s\upharpoonright n}$ is the desired injection from $2^\omega$ into the set of alternating endpoints of $X$. 

    The argument for red and green endpoints is very similar. For example to produce uncountably many red points we follow the same construction, with the only difference that at every step we choose $x_s$ to be red and to be such that the segment from $x_s$ to $y_s$ only contains red points.
\end{proof}

We now start investigating the dynamical properties of the action $\Homeo(X)\curvearrowright X$, in particular we will show that the action is minimal and not $\tame_1$. In Section \ref{section: open questions and conclusion} we will point out some differences between this action and the action $\Homeo(W_P)\curvearrowright W_P$. Many of the arguments concerning $W_P$ in the previous section were based on Proposition \ref{prop: characterization of W_P}, so we begin by proving an analogous result for the dendrite $X$.

\begin{lemma}\label{lemma: pairs of alternating endpoints}
    Let $E=\{e_1,e_2,\ldots\}$ and $F=\{f_1,f_2,\ldots\}$ be two countable dense subsets of $X$ consisting of alternating endpoints. Then there is $h\in\Homeo(X)$ such that $h(E)=F$. Moreover $h$ can be chosen so that $h(e_1)=f_1$ and $h(e_2)=f_2$.
\end{lemma}
\begin{proof}
    The argument is the same as the proof of Theorem 6.2 of \cite{CD}, where a similar statement is proved for the \Wazewski dendrites. The only extra step required here is the trivial observation that if $x_1,x_2,y_1,y_2\in X$ are endpoints such that $x_1$ has the same type as $y_1$ and $x_2$ has the same type as $y_2$ (but the types of $x_1$ and $x_2$ could differ), then there is a homeomorphism $g\colon [x_1,x_2]\to[y_1,y_2]$ such that $g(x_1)=y_1$, $g(x_2)=y_2$, and $h$ preserves the order of ramification points. A similar statement holds for arcs $[e_1,x_1]$ and $[e_2,x_2]$, where $e_1,e_2$ are alternating endpoints and $x_1,x_2$ are ramification points of the same colour. We spell out some details for completeness. 
    
    Let $I_1=[e_1,e_2]$, $J_1=[f_1,f_2]$ and let $h\colon I_1\to J_1$ be a homeomorphism such that $h(e_i)=f_i$ for $i=1,2$ and such that for all ramification points $x$ in $I_1$, $h(x)$ is a ramification point of the same order in $J_1$. We now do a back and forth argument to ensure that every element of $E$ is mapped by $h$ to an element of $F$ in the forward step, while also guaranteeing that every element of $F$ is the image of some element of $E$ in the backward step. We start with a backward step. Let $J_2$ be the shortest arc between $f_3$ and $J_1$, meaning that $J_1\cap J_2=\{s_1\}$ is a singleton. Let $r_1$ be the unique point in $J_1$ with $h(r_1)=s_1$ and note that, by construction, $r_1$ and $s_1$ are ramification points of the same order. Let $n_1$ be the smallest natural number such that $e_{n_1}$ meets $J_1$ in $r_1$, which exists by density of $E$. Let $I_2=[e_{n_1},r_1]$ and define $h\colon I_2\to J_2$ to be a homeomorphism so that $h(e_{n_1})=f_3$ and $h$ maps ramification points in $I_2$ to ramification points of the same order in $J_2$. Now we do a forward step. Let $n_2$ be the smallest natural number such that $e_{n_2}$ is not already in the image of $h$, let $I_3$ be the shortest arc between $e_{n_2}$ and $I_1\cup I_2$, meaning that $I_3\cap(I_1\cup I_2)=\{r_2\}$ is a singleton. Note that there are various possibilities for the position of $r_2$ on $I_1\cup I_2$, but in any case $h$ has already been defined on $r_2$, so let $s_2=h(r_2)$, which is a ramification point in $J_1\cup J_2$ with the same order as $r_2$. Now define $h(e_{n_2})$ to be $f_{m_1}$, where $m_1$ is the smallest index for which $f_{m_1}$ is not already in the image of $h$ and such that $s_2$ is the only point in $[f_{m_1},s_2]\cap(J_1\cup J_2)$, which exists since $F$ is dense in $X$. Let $J_3=[f_{m_1},s_2]$ and define $h\colon I_3\to J_3$ to be a homeomorphism which sends ramification points of $I_3$ to ramification points of the same order in $J_3$. For the next backward step we can let $f_{m_2}$ be the smallest index element of $F$ not already in the image of $h$ and $J_4$ be the shortest arc between $f_{m_2}$ and $J_1\cup J_2\cup J_3$, meaning that $J_4\cap(J_1\cup J_2\cup J_3)=\{s_3\}$ is a singleton. Let $r_3$ to be the only point in $I_1\cup I_2\cup I_3$ with $h(r_3)=s_3$. We now let $e_{n_3}$ to be the smallest index element of $E$ which is not already in the domain of $h$ and such that $r_3$ is the only point in $[e_{n_3},r_3]\cap(I_1\cup I_2\cup I_3)$. Let $I_4=[e_{n_3},r_3]$. We define $h\colon I_4\to J_4$ to be a homeomorphism such that preserves the order of ramification points and such that $h(e_{n_3})=f_{m_2}$. By iterating this construction we define $I_5,J_5,I_6,J_6,I_7,J_7,\ldots$ and we obtain a homeomorphism $h\colon\bigcup I_n\to\bigcup J_n$. Since $h$ is defined on a dense set and preserves the order of ramification points, it can be extended to a homeomorphism $h\colon X\to X$. In the last step we use the fact that the topology of $\Homeo(X)$ coincides with the one obtained by embedding $\Homeo(X)$ as a subgroup of $\mathrm{Sym}(X_{\mathrm{ram}})$ \cite[Proposition 2.4]{duchesneMonod}, where $X_{\mathrm{ram}}$ is the set of ramification points of $X$, $\mathrm{Sym}$ denotes its permutation group with the pointwise convergence topology and the embedding $\Homeo(X)\to\mathrm{Sym}(X_{\mathrm{ram}})$ is given by $h\mapsto h\upharpoonright X_{\mathrm{ram}}$.
\end{proof}

\begin{prop}\label{prop: the action of Homeo(X) on X is minimal} 
    The action $\Homeo(X)\curvearrowright X$ is minimal.
\end{prop}
\begin{proof}
    Note that the same argument as in the proof of Lemma \ref{lemma: pairs of alternating endpoints} shows that whenever $e_1,e_2$ and $f_1,f_2$ are endpoints of $X$ such that $e_1,f_1$ are alternating while $e_2$ and $f_2$ are either both red or both green, then there exists $h\in\Homeo(X)$ with $h(e_1)=f_1$ and $h(e_2)=f_2$. This in particular implies that the action $\Homeo(X)\curvearrowright\End(X)$ has three orbits, those being the three kind of endpoints distinguished in Definition \ref{def: three kind of endpoints}. As those three orbits are all dense in $X$, we only need to show that the orbits of regular and ramification points are dense. Fix $x\in X\setminus\End(X)$, $y\in X$ and $\varepsilon>0$. We want to find $h\in\Homeo(X)$ with $d(h(x),y)<\varepsilon$. Let $z\in X\setminus\End(X)$ be such that $d(z,y)<\varepsilon/3$ (if $y$ is not an endpoint we can take $z=y$). We will construct $h\in\Homeo(X)$ such that $d(h(x),z)\leq\varepsilon/3$, and conclude by the triangle inequality. Using the fact that alternating points are dense in $X$, let $e_1,e_2$ be alternating endpoints with $x\in[e_1,e_2]$ and let $f_1,f_2$ be alternating endpoints with $z\in[f_1,f_2]$ and $\diam([f_1,f_2])<\varepsilon/6$. By Lemma \ref{lemma: pairs of alternating endpoints} there is $h\in\Homeo(X)$ with $h(e_1)=f_1$ and $h(e_2)=f_2$. Since $h(x)\in[f_1,f_2]$ we have $d(h(x),z)\leq\varepsilon/3$, as needed.
\end{proof}

\begin{thm}
    The action $\Homeo(X)\curvearrowright X$ is not $\tame_1$.
\end{thm}
\begin{proof}
    We want to apply Proposition \ref{prop: Ellis} again, so we first need to verify that the action of $\Homeo(X)$ on $X$ is proximal. By Theorem 10.1 of \cite{duchesne2018group} any minimal action on a dendrite is (strongly) proximal which, together with Proposition \ref{prop: the action of Homeo(X) on X is minimal} shows that $\Homeo(X)\curvearrowright X$ is proximal. Now given $a,b$ alternating endpoints in $X$, we will show that the idempotent $p_{a,b}$ belongs to $E(X,\Homeo(X))$, which, in conjunction with Lemma \ref{lemma: uncountably many endpoints of each type}, is enough to conclude in the same fashion as in the proof of Theorem \ref{thm: Wazewski is not tame1}. Fix $a,b$ alternating endpoints in $X$ and fix sequences $(a_n)_{n\in\omega}$ and $(b_n)_{n\in\omega}$ such that 
    \begin{itemize}
        \item For all $n$, both $a_n$ and $b_n$ are ramification points of $X$.
        \item For all even $n$, both $a_n$ and $b_n$ have order $4$, while for all odd $n$, both $a_n$ and $b_n$ have order $3$.
        \item We have $\lim_{n\to\infty}a_n=a$ and $\lim_{n\to\infty}b_n=b$.
        \item For all $n$, $a,a_n,b_n$ and $b$ appear in $[a,b]$ in this order.
    \end{itemize}
    We claim that for all $n$ there exists a homeomorphism $h_n\colon X\to X$ such that $h_n(a)=a$, $h_n(b)=b$, $h_n(b_n)=a_n$, indeed this can be constructed as in the proof on Lemma \ref{lemma: pairs of alternating endpoints} by letting $e_1=f_1=a$, $e_2=f_2=b$ so that $I_1=J_1=[a,b]$, and choosing the homeomorphism $h\colon I_1\to J_1$ in the first step so that $h(b_n)=a_n$, which is clearly possible. We now have $p_{a,b}=\lim h_n$ (in the pointwise convergence topology), showing that $p_{a,b}\in E(X,\Homeo(X))$, concluding the proof.
\end{proof}

\section{Actions on dendrites cannot have high \texorpdfstring{$\beta$}{beta}-rank}
\label{section: actions on dendrites have low beta-rank}
The notion of $\beta$-rank of a dynamical system was introduced by Glasner and Megrelishvili in \cite{GM}, based on the definition of $\beta$-rank for a real-valued Baire class 1 function from Bourgain \cite{Bourgain}. It follows from results of Kechris and Louveau in \cite{KechrisLouveau} that $\beta(X,G)\leq\omega_1$ for a metric tame system, but all the examples given in \cite{GM} have $\beta$-rank at most $2$ and the authors asked for examples of tame metric systems with higher $\beta$-rank. In this section we show that if $(X,G)$ is a dynamical system, with $X$ a dendrite, then $\beta(X,G)\leq 2$. In the following section we produce, for every $\alpha<\omega_1$, a metric tame dynamical system $(X_\alpha,G_\alpha)$ with $\beta(X_\alpha,G_\alpha)=\alpha$.

\begin{defn}
Let $(X,G)$ be a metric tame dynamical system. For $p\in E(X,G)$ define the \emph{oscillation of $p$ at $x$ to be} $$\osc(p,x)=\inf\{\diam(pV)\mid V \text{is an open neighbourhood of }x\},$$
and relativize this notion to subsets $A\subseteq X$ with $x\in A$ by setting $$\osc(p,x,A)=\osc(p\upharpoonright A,x).$$
Given $\varepsilon>0$ define a derivative operation $$A\mapsto A^1_{\varepsilon,p}=\{x\in X\mid \osc(p,x,A)\geq\varepsilon\},$$
and iterating through countable ordinals define $A^{\alpha}_{\varepsilon,p}$ for all $\alpha<\omega_1$, where $A^\gamma_{\varepsilon,p}=\bigcap_{\beta<\gamma}A^\beta_{\varepsilon,p}$ for a limit ordinal $\gamma$.
Let $$\beta(p,\varepsilon,A)=\begin{cases}\text{the least $\alpha$ with } A^{\alpha}_{\varepsilon,p}=\varnothing & \text{if such an $\alpha$ exists}\\
\omega_1 & \text{otherwise}
\end{cases}.$$
Define $\beta(p,\varepsilon)=\beta(p,\varepsilon,X)$ and let the \emph{oscillation rank} of $p$ be $$\beta(p)=\sup_{\varepsilon>0}\beta(p,\varepsilon).$$

Finally let $$\beta(X,G)=\sup\{\beta(p)\mid p\in E(X,G)\}.$$
\end{defn}

We want to show that if $X$ is a dendrite, $G$ is a topological group acting continuously on $X$, $\varepsilon>0$ and $f\in E(X,G)$, then  $X^2_{\varepsilon,f}=\varnothing$. We will show a stronger statement, namely that $X^1_{\varepsilon,f}$ is finite hence discrete, from which the desired conclusion follows immediately. Clearly it suffices to consider the $G=\Homeo(X)$ case.

\begin{lemma}\label{lemma: elements of E(X,G) preserve betweenness}
    Let $(X,G)$ be a dynamical system, where $X$ is dendrite. Any $f\in E(X,G)$ preserves the betweenness relation on $X$, meaning that \begin{equation}z\in[x,y]\implies f(z)\in[f(x),f(y)].\label{eq: preserving betweenness}\end{equation}
\end{lemma}

By \cite[Proposition 19.9]{GMtrees} the betweenness relation on a dendrite is closed in $X^2$, from which the above lemma follows immediately. We give an alternative direct proof below.

\begin{proof}
    We show that the set of functions not satisfying \eqref{eq: preserving betweenness} is open in $X^X$. Since the elements of $G$ clearly satisfy \eqref{eq: preserving betweenness} and $E(X,G)$ is the pointwise closure of $G$ in $X^X$, this is enough to finish the proof. Let $g\in X^X$ be a function not satisfying \eqref{eq: preserving betweenness}, suppose for example that there are $x,y,z\in X$ with $x\in[y,z]$ but $g(x)\not\in[g(y),g(z)]$. Let $\varepsilon>0$ be small enough so that \begin{align*}
        B_\varepsilon(g(x))\cap[g(y),g(z)]&=\varnothing\\
    \end{align*}
    and the analgous intersection is empty when swapping the roles of $x,y,z$. Any $h\in X^X$ which is $\varepsilon$-close to $g$ on $x,y,z$ will also satisfy $h(x)\not\in[h(y),h(z)]$. 
\end{proof}

Since $X$ is a dendrite, given any $\varepsilon>0$ fix a finite open cover $\mathcal U_\varepsilon=\{U_0,\ldots,U_{n(\varepsilon)}\}$ such that (1) $\diam(U_i)\leq\varepsilon/4$ for all $i$. (2) $\partial U_i$ is finite for every $i$ and (3) every $U_i$ is convex, meaning that $x,y\in U_i\implies[x,y]\subseteq U_i$.

\begin{defn}
    Let $x\in X$, $f\in E(X,G)$, $\varepsilon>0$ and $i$ be such that $f(x)\in U_i$. Given $y\in\partial U_i$ we say that $f$ is  \emph{$\varepsilon$-oscillating in the $y$-direction at $x$} if for every open neighbourhood $V$ of $x$, there exists $x^V\in V$ such that \begin{itemize}
        \item $d(f(x),f(x^V))\geq\varepsilon$,
        \item $y\in[f(x),f(x^V)].$
    \end{itemize}
\end{defn}
Note that it is possible for an $f\in E(X,G)$ to be $\varepsilon$-oscillating in many directions at the same time at some $x\in X$.

\begin{lemma}\label{lemma: being in first derivative implies oscillation}
    If $x\in X^1_{\varepsilon,f}$ and $f(x)\in U_i$, then there is at least one $y\in \partial U_i$ such that $f$ is $\varepsilon/2$-oscillating in the $y$-direction at $x$.    
\end{lemma}
\begin{proof}
    Since $x\in X^1_{\varepsilon,f}$, for every open neighbourhood $V$ of $x$ there are $z^V,y^V\in V$ with $d(f(z^V),f(y^V))\geq\varepsilon$. By the triangle inequality we must have either $d(f(x),f(z^V))\geq\varepsilon/2$ or $d(f(x),f(y^V))\geq\varepsilon/2$. In the first case let $x^V=z^V$ while in the latter case let $x^V=y^V$. Then $f(x^V)\not\in U_i$ and, for some $w\in\partial U_i$, $w\in[f(x),f(x^V)]$. By the pigeonhole principle, since $\partial U_i$ is finite, there must at least one $y\in\partial U_i$ so that for arbitrarily small open $V\ni x$, $y\in[f(x),f(x^V)]$, so that $f$ is $\varepsilon/2$-oscillating in the $y$-direction at $x$. 
\end{proof}

\begin{lemma}\label{lemma: oscillation must be in different directions}
    If $x_1\neq x_2\in X^1_{\varepsilon,f}$ and $i$ are such that $f(x_1),f(x_2)\in U_i$, then $f$ cannot be $\varepsilon/2$-oscillating in the same direction at $x_1$ and $x_2$.
\end{lemma}
\begin{proof}
    Suppose for a contradiction that there is $y\in\partial U_i$ such that $f$ is $\varepsilon/2$-oscillating in the $y$-direction at both $x_1$ and $x_2$. Let $V$ be an open neighbourhood of $x_1$ with $x_2\not\in V$ and let $x_1^V$ witness that $f$ is $\varepsilon/2$-oscillating in the $y$ direction at $x_1$. Let $w$ be the midpoint of $x_1,x_2$ and $x_1^V$, meaning that $$w=[x_1,x_2]\cap[x_1,x_1^V]\cap[x_1^V,x_2].$$
    We distinguish three cases, based on the position of $w$ relative to $x_1,x_2,x_1^V$:
    \begin{itemize}
        \item[Case 1:] $w=x_1^V$. This implies that $x_1^V\in[x_1,x_2]$. Since $f$ preserves betweenness, we obtain $$f(x_1^V)\in[f(x_1),f(x_2)]\subseteq U_i,$$ where the last inclusion holds by convexity of $U_i$. Together with $\diam(U_i)\leq\varepsilon/4$, this contradicts that $x_1^V$ witnesses that $f$ is $\varepsilon/2$-oscillating in the $y$ direction at $x_1$.
        \item[Case 2:] $w=x_1$. Let $U$ be an open neighbourhood of $x_2$ small enough to have $U\cap V=\varnothing$, and let $x_2^U$ witness that $f$ is $\varepsilon/2$-oscillating in the $y$ direction at $x_2$. Since  $x_1\in[x_1^V,x_2]$ by assumption, we also have $x_1\in[x_1^V,x_2^U]$, which implies $f(x_1)\in[f(x_1^V),f(x_2^U)]$. This is a contradiction since $f(x_1)\in U_i$ while the arc $[f(x_1^V),f(x_2^U)]$ is entirely outside of $U_i$ (since $x_1^V$ and $x_2^U$ are witnesses for the $\varepsilon/2$-oscillation of $f$ at $x_1$, $x_2$ respectively, this arc can at most go through $y$, but must remain outside of $U_i$.)     
        \item[Case 3:] $w\in[x_1,x_2]$ but $w\neq x_1^V,x_1$. Let $U$ be an open neighbourhood of $x_2$ small enough to have $U\cap V=\varnothing$, and let $x_2^U$ witness that $f$ is $\varepsilon/2$-oscillating in the $y$ direction at $x_2$. By construction we have $w\in[x_1^V,x_2^U]$, which implies $f(w)\in[f(x_1^V),f(x_2^U)]$. This is a contradiction, as we have $f(w)\in U_i$ since $f$ preserves betweenness, while the arc $[f(x_1^V),f(x_2^U)]$ is entirely outside of $U_i$ as in the previous case.
    \end{itemize}
\end{proof}

Putting together the previous lemmas we finally obtain 
\begin{thm}\label{thm: actions on dendrites have low beta-rank}
    Let $f\in E(X,G)$. Then $X^1_{\varepsilon,f}$ is finite. In particular $\beta(f)\leq 2$ and $\beta(X,G)\leq 2$.
\end{thm}
\begin{proof}
    By Lemma \ref{lemma: being in first derivative implies oscillation} if $x\in X^1_{\varepsilon,f}$ and $f(x)\in U_i$, then $f$ is $\varepsilon/2$-oscillating in the $y$-direction for some $y\in\partial U_i$. Since $\partial U_i$ is finite there can only be finitely many $x\in X^1_{\varepsilon,f}$ with $f(x)\in U_i$ by Lemma \ref{lemma: oscillation must be in different directions}, so that $X^1_{\varepsilon,f}$ is finite since $\U_\varepsilon$ is finite. Since $\varepsilon$ was arbitrary we have $\beta(f)\leq 2$ and since $f$ was arbitrary we have $\beta(X,G)\leq 2.$
\end{proof}

\section{Metric tame systems with arbitrary \texorpdfstring{$\beta$}{beta}-rank}
\label{section: systems with arbitrary beta-rank}
In this section we build, for every $\alpha<\omega_1$, a tame metric system $(X_\alpha,G_\alpha)$ with $\beta(X_\alpha,G_\alpha)=\alpha$, answering Question 11.8 of \cite{GM}. The intuitive idea is that $X_\alpha$ will be a set of Cantor-Bendixson rank $\alpha$, together with a continuous action of a group $G_\alpha$ constructed in such a way that for some $f\in E(X_\alpha,G_\alpha)$, the derivative operations used to compute the Cantor-Bendixson rank of $X_\alpha$ and the $\beta$-rank of $f$ coincide. Since the spaces $X_\alpha$ will be countable, the dynamical system $(X_\alpha,G_\alpha)$ is hereditarily nonsensitive by \cite[Corollary 10.2]{GMHNSsystems}, a property strictly stronger than tameness. We begin by recalling some basic facts about the Cantor-Bendixson rank.

\begin{defn}
    Let $X$ be a topological space. Define $$X^1=\{x\in X\mid x\text{ is a limit point of }X\}$$ and recursively define, for ordinals $\alpha$, 
    \begin{align*}
        X^{\alpha+1}&=(X^\alpha)^1 \\
        X^\beta&=\bigcap_{\alpha<\beta}X^\alpha\quad\text{for limit $\beta$.}
    \end{align*}
    The least $\alpha$ for which $X^\alpha=X^{\alpha+1}$ is called the Cantor-Bendixson rank of $X$, denoted by $|X|_{CB}$.
\end{defn}

It is well known that if $X$ is a Polish (or more generally a second countable) space, then $|X|_{CB}<\omega_1$, and that for any $\alpha<\omega_1$, there is a compact countable $Y_\alpha\subseteq[0,1]$ with $|Y_\alpha|_{CB}=\alpha$, see \cite[Chapter 6.C]{kechrisDST} for more details.

Consider now the space $X_\alpha=Y_\alpha\times\{0,1\}$ (as a subspace of $[0,1]^2$), and define $G_\alpha=\Homeo(X_\alpha)$. With the compact-open topology, $G_\alpha$ is a Polish group with a natural continuous action $G_\alpha\curvearrowright X_\alpha$. Note that $|X_\alpha|_{CB}=|Y_\alpha|_{CB}=\alpha$.

\begin{lemma}
    For all $\alpha<\omega_1$, the dynamical system $(X_\alpha,G_\alpha)$ described above is tame.
\end{lemma}

As mentioned above the spaces $X_\alpha$ are countable so the dynamical system $(X_\alpha,G_\alpha)$ is hereditarily nonsensitive, which is strictly stronger than being tame. We give a short proof of tameness below for completeness.

\begin{proof}
    Since $X_\alpha$ is countable, any $f\in X_\alpha^{X_\alpha}$ has the property that $f^{-1}(A)$ is $\mathbf{\Sigma}^0_2$ whenever $A$ is open. In other words every such $f$ and in particular every $f\in E(X_\alpha,G_\alpha)$ is Baire class 1. By \cite[Theorem 6.3]{GlasnerMegrelishviliUspenskij} this is equivalent to the tameness of $(X_\alpha,G_\alpha)$.
\end{proof}

\begin{prop}\label{prop: rank is at most alpha}
    Fix $\alpha<\omega_1$ and let $(X_\alpha,G_\alpha)$ be as above. We have $\beta(X_\alpha,G_\alpha)\leq|X_\alpha|_{CB}=\alpha$.
\end{prop}
\begin{proof}
    Fix $f\in E(X_\alpha,G_\alpha)$ and $\varepsilon>0$. We show by induction that $(X_\alpha)^\beta_{\varepsilon,f}\subseteq(X_\alpha)^\beta$ for all $\beta\leq\alpha$. Since the points of $X_\alpha\setminus(X_\alpha)^1$ are isolated in $X_\alpha$, there is an inclusion $(X_\alpha)^1_{\varepsilon,f}\subseteq(X_\alpha)^1$. The same observation, together with the inductive hypothesis $(X_\alpha)^\beta_{\varepsilon,f}\subseteq(X_\alpha)^\beta$, shows that $(X_\alpha)^{\beta+1}_{\varepsilon,f}\subseteq(X_\alpha)^{\beta+1}$. If instead $\gamma$ is a limit ordinal, the inclusion follows immediately since both $(X_\alpha)^\gamma$ and $(X_\alpha)^\gamma_{\varepsilon,f}$ are defined as the intersection of the sets of lower index in the respective hierarchies. We thus have $(X_\alpha)^\beta_{\varepsilon,f}\subseteq(X_\alpha)^\beta$ for all $\beta\leq\alpha$. It's clear from the definition of the Cantor-Bendixson rank that $X^{|X|_{CB}}$ is perfect, and since $X_\alpha$ is countable this implies that $(X_\alpha)^\alpha=\varnothing$. By the inclusion above we obtain $(X_\alpha)^\alpha_{\varepsilon,f}=\varnothing$, or in other words $\beta(f,\varepsilon)\leq\alpha$. Since $\varepsilon$ was arbitrary we have $\beta(f)\leq\alpha$, and since $f$ was arbitrary we also obtain $\beta(X_\alpha,G_\alpha)\leq\alpha$.
\end{proof}

Recall that every ordinal $\alpha$ can be written in the form $\alpha=\beta+n$, where $\beta$ is a limit ordinal and $n$ is a finite ordinal (where $n=0$ is possible). The \emph{parity of $\alpha$} is defined to be the parity of $n$ in the decomposition above.

\begin{prop}\label{prop: rank is at least alpha}
    Fix $\alpha<\omega_1$ and let $(X_\alpha,G_\alpha)$ be as above. We have $\beta(X_\alpha,G_\alpha)\geq\alpha$.
\end{prop}

\begin{proof}
    It suffices to construct an $f\in E(X_\alpha,G_\alpha)$ with $\beta(f)\geq\alpha$. For any $x\in X_\alpha$, let $\rk(x)$ be the least ordinal $\beta\leq\alpha$ such that $x\in(X_\alpha)^\beta$ but $x\not\in(X_\alpha)^{\beta+1}$. Define $f$ by $$f((y,i))=\begin{cases}
        (y,i) & \text{if $\rk(y)$ is even} \\
        (y,1-i) & \text{if $\rk(y)$ is odd.}
    \end{cases}$$
For any $0<\varepsilon<1$ we now have $(X_\alpha)^1_{\varepsilon,f}=(X_\alpha)^1$. We have the $\subseteq$ inclusion as in the proof of Proposition \ref{prop: rank is at most alpha}. For the reverse inclusion we prove that for every $x\in(X_\alpha)^1$, we can find a sequence $(x_i)_{i<\omega}$ of distinct elements of $X_\alpha$ such that $x_i\to x$, $\rk(x_i)<\rk(x)$, and $\rk(x_i)$ and $\rk(x)$ have opposite parity for all $i$. We proceed by induction on $\rk(x)<\alpha$.
\begin{itemize}
    \item If $\rk(x)=1$, then $x\in(X_\alpha)^1$ but $x\not\in(X_\alpha)^2$, meaning that $x$ is isolated in $(X_\alpha)^1$ but not in $X_\alpha$. This implies that there exists a sequence $(x_i)_{i<\omega}\subseteq X_\alpha\setminus(X_\alpha)^1$ with $x_i\to x$. Since $\rk(x_i)=0$ for every $i$, the base case is done.
    \item If $\rk(x)=\beta+1$ is a successor ordinal, the same argument as in the previous case gives a sequence $(x_i)_{i<\omega}\subseteq (X_\alpha)^\beta\setminus(X_\alpha)^{\beta+1}$ with $x_i\to x$. Since $\rk(x_i)=\beta$, while $\rk(x_i)=\beta+1$, this establishes the successor case.
    \item If instead $\rk(x)=\gamma$ is a limit ordinal, then there must be a sequence $(x_i)_{i<\omega}$ with $\rk(x_i)<\gamma$ for every $i$ with $x_i\to x$. If $\rk(x_i)$ and $\rk(x)$ have the same parity, we can simply replace $x_i$ with a good enough approximation of the opposite rank, which exists by inductive hypothesis.
\end{itemize}
 Now given any open neighbourhood $U$ of $x$ there is $x_i\in U$ with $d(f(x),f(x_i))\geq\varepsilon$, showing that the oscillation of $f$ at $x$ is at least $\varepsilon$. This shows $(X_\alpha)^1\subseteq(X_\alpha)^1_{\varepsilon,f}$.
 We now prove by induction that $(X_\alpha)^\beta_{\varepsilon,f}=(X_\alpha)^\beta$ for all $\beta<\alpha$.

 At a successor ordinal $\beta+1$ we can find, for every $x\in(X_\alpha)^{\beta+1}=\left((X_\alpha)^\beta\right)^1$, a sequence $x_i\in(X_\alpha)^\beta$ with $x_i\to x$, as we have shown in the successor case of the previous argument. This shows that $x\in\left((X_\alpha)^\beta\right)^1_{\varepsilon,f}$, which by inductive hypothesis is the same as $$x\in\left((X_\alpha)^\beta_{\varepsilon,f}\right)^1_{\varepsilon,f}=(X_\alpha)^{\beta+1}_{\varepsilon,f},$$ showing $(X_\alpha)^{\beta+1}_{\varepsilon,f}\supseteq(X_\alpha)^{\beta+1}$. The other inclusion follows as in Proposition \ref{prop: rank is at most alpha}.
 
 At a limit ordinal $\gamma$ the result follows immediately from the fact that both $(X_\alpha)^\gamma$ and $(X_\alpha)^\gamma_{\varepsilon,f}$ are defined as the intersection of the sets of lower ranks in the respective hierarchies.
 
We have thus obtained $(X_\alpha)^\beta_{\varepsilon,f}=(X_\alpha)^\beta$ for all $\beta<\alpha$. Since $(X_\alpha)^\beta\neq\varnothing$, we also have $(X_\alpha)^\beta_{\varepsilon,f}\neq\varnothing$, showing that $\beta(f)\geq\alpha$. By Proposition \ref{prop: rank is at most alpha}, this means that $\beta(f)=\alpha$.

It remains to verify that $f\in E(X_\alpha,G_\alpha)$. In other words given finitely many $(x_1,i_1),\ldots,(x_n,i_n)\in X_\alpha$ for some $n<\omega$ and $i_j\in\{0,1\}$, we want to find $g\in G_\alpha$ such that $g(x_i)=f(x_i)$ for all $i$. Since $Y_\alpha$ is zero-dimensional we can find pairwise disjoint clopen neighbourhoods $U_j$ of $x_j$, so that $U_j\times\{i_j\}$ are clopen neighbourhoods of $(x_j,i_j)$ with pairwise disjoint projections. Thanks to the latter property the map $g$ defined by mapping $U_j\times\{i_j\}$ to $U_j\times\{1-i_j\}$ whenever $f((x_j,i_j))=(x_j,1-i_j)$ and by the identity otherwise, is an element of $G_\alpha$ in the given neighbourhood of $f$, showing that $f\in E(X_\alpha,G_\alpha)$.
\end{proof}

Combining the two previous propositions we immediately obtain

\begin{thm}
    Let $\alpha<\omega_1$ and $(X_\alpha,G_\alpha)$ as above. Then $\beta(X_\alpha,G_\alpha)=\alpha$.
\end{thm}

\begin{remark}
A simple modification of the examples above can be used to produce, for every $\alpha<\omega_1$ a tame, metric system $(Z_\alpha,H_\alpha)$ with $Z_\alpha$ connected and $\beta(Z_\alpha,H_\alpha)=\alpha$. Indeed fix $\alpha<\omega_1$ and let $(X_\alpha,G_\alpha)$ be as above. Let $Z_\alpha=C(X_\alpha)$ be the cone over $X_\alpha$, that is the quotient space of $X_\alpha\times[0,1]$ by the relation identifying $(x,t)$ and $(x',t')$ if and only if $t=t'=1$. Clearly $Z_\alpha$ is (path) connected. Consider $H_\alpha=G_\alpha\times\{\Id_{[0,1]}\}$ as a subgroup of $\Homeo(X_\alpha\times[0,1])$. It is easy to check that $E(X_\alpha\times[0,1],H_\alpha)\cong E(X_\alpha,G_\alpha)$, so that the system $(X_\alpha\times[0,1],H_\alpha)$ is tame. The quotient is a factor map $(X_\alpha\times[0,1],H_\alpha)\to(Z_\alpha,H_\alpha)$, so the latter system is also tame, being a factor of a tame system \cite[Proposition 6.4]{KerrLi}. The same arguments as in the proof of Proposition \ref{prop: rank is at most alpha} and Proposition \ref{prop: rank is at least alpha} now show that $\beta(Z_\alpha,H_\alpha)=\alpha$.
\end{remark}

\section{Some open questions}
\label{section: open questions and conclusion}

Based on the examples of Section \ref{section: Wazewski dendrites have actions that are not tame_1} and \ref{section: the second example} we ask the following question.

\begin{question}
    Let $X$ be a dendrite such that the action $\Homeo(X)\curvearrowright X$ is minimal. Can this action be $\tame_1$? More generally is there a characterization of the dendrites $X$ for which the action $\Homeo(X)\curvearrowright X$ is $\tame_1$, or at least a natural class of dendrites, such that any action on a dendrite in this class is $\tame_1$?
\end{question}

While investigating this question, with the goal of using Proposition \ref{prop: Ellis} and an argument similar to those in Section \ref{section: Wazewski dendrites have actions that are not tame_1} and \ref{section: the second example} to show that such an action is not $\tame_1$, we wondered about two related questions, namely:
\begin{itemize}
    \item Let $X$ be a dendrite such that the action $\Homeo(X)\curvearrowright X$ is minimal. Must the induced diagonal action $\Homeo(X)\curvearrowright X^2\setminus\Delta_X$ also be minimal?
    \item Let $X$ be a dendrite such that the action $\Homeo(X)\curvearrowright X$ is minimal. Let $e,f\in\End(X)$ and let $\mathrm{Stab}_{e,f}\subseteq\Homeo(X)$ be the pointwise stabilizer of $\{e,f\}$. Must the action $\mathrm{Stab}_{e,f}\curvearrowright [e,f]$ have a dense orbit?
\end{itemize}

While the answer is positive for the action of $\Homeo(W_P)$ on $W_P$, the dendrite $X$ constructed in Section \ref{section: the second example} provides a counterexample to both questions. Let $e,f\in\End(X)$ be endpoints such that $e$ is green, $f$ is red, and the arc $[e,f]$ only alternates between green and red once. Then the action $\mathrm{Stab}_{e,f}\curvearrowright[e,f]$ has no dense orbit. Indeed the orbit of no point in the green part of the arc can approximate $f$, while the orbit of no point in the red part of the arc can approximate $e$. It is also easy to check that the orbit of $(e,f)$ under the diagonal action cannot approximate a pair of alternating endpoints, answering negatively the first part of the question as well.

Based on the examples constructed in Section \ref{section: systems with arbitrary beta-rank} the following question is natural.

\begin{question}
    Are there examples, for every $\alpha<\omega_1$, of tame, metric, minimal systems $(M_\alpha,G_\alpha)$ with $\beta(M_\alpha,G_\alpha)=\alpha$? Are there such examples where the action is on a connected space?
\end{question}

\section*{Ackowledgements}
We would like to thank Aleksandra Kwiatkowska for many useful conversations about dendrites and proofreading an early version of this paper. We would also like to thank Eli Glasner for a nice conversation about tame dynamical systems and their $\beta$-rank at the 13th Prague symposium in topology, as well as Michael Megrilishvili for reading an early version of this paper and pointing out both the easier proof of Lemma \ref{lemma: elements of E(X,G) preserve betweenness} and the fact that the dynamical systems constructed in Section \ref{section: systems with arbitrary beta-rank} are not only tame, but also hereditarily nonsensitive. We would like to thanks the anonymous referee for their comments, which improved the exposition in the paper.

\end{document}